\documentclass[twoside,leqno,11pt]{article}

\usepackage{amssymb,amsfonts,amsmath}
\usepackage[mathscr]{eucal}
\usepackage[english]{babel}
\usepackage[latin1]{inputenc}
\usepackage{graphicx}
\usepackage{graphics}
\usepackage{color}

\newtheorem{theorem}{Theorem}
\newtheorem{corollary}[theorem]{Corollary}
\newtheorem{definition}[theorem]{Definition}

\newtheorem{lemma}[theorem]{Lemma}

\newtheorem{proposition}[theorem]{Proposition}

\newenvironment{proof}[1][Proof]{\noindent\textbf{#1.} }{\ \rule{0.5em}{0.5em}}

\def\dom{\mathop{\mathrm{dom}}\nolimits}
\def\im{\mathop{\mathrm{im}}\nolimits}

\def\rank{\mathop{\mathrm{rank}}\nolimits}
\def\id{\mathop{\mathrm{id}}\nolimits}

\def\R{\mathbb R}
\def\T{\mathcal{T}}

\def\O{\mathcal{O}}
\def\OP{\mathcal{OP}}

\title{On Relative Ranks of the Semigroup of Orientation-preserving Transformations on Infinite Chains}

\author{Ilinka Dimitrova\footnote{The author gratefully acknowledges support of DAAD, within the funding programme "Research Stays for University Academics and Scientists, 2019" number 57442043.} \\
\textit{Faculty of Mathematics and Natural Science}\\
\textit{South-West University "Neofit Rilski"}\\
\textit{2700 Blagoevgrad, Bulgaria}\\
\textit{e-mail: ilinka\_dimitrova@swu.bg}\\
~~\\
J\"{o}rg Koppitz \\
\textit{Institute of Mathematics}\\
\textit{Potsdam University}\\
\textit{Potsdam, 14469, Germany}\\
\textit{koppitz@uni-potsdam.de}\\
~~\\
\textit{Institute of Mathematics and Informatics}\\
\textit{Bulgarian Academy of Sciences}\\
\textit{1113 Sofia, Bulgaria}\\
\textit{e-mail: koppitz@math.bas.bg}}

\begin{document}

\maketitle

\begin{abstract}
In this paper, we determine the relative rank of the semigroup
$\OP(X)$ of all orientation-preserving transformations on infinite chains modulo the semigroup $\O(X)$ of all order-preserving transformations.
\end{abstract}

\textit{Key words:} transformation semigroups on infinite chains, order-preserving
transformations, orientation-preserving transformations, relative rank.

2010 Mathematics Subject Classification: 20M20
\\

\section{Introduction and Preliminaries}

The \textit{rank} of a semigroup $S$ is the minimum cardinality of a
generating set of $S$. For a finitely generated semigroup $S$ determining the rank of $S$ is a natural question
and obviously yields some information about the semigroup.
The ranks of various well known semigroups have been calculated \cite{FS, GH, GH1, HMcF}.
By the famous theorem of Sierpi\'{n}ski \cite{Sierp}, we know that any countable semigroup
can be embedded in a semigroup with rank less than three. On the other hand,
if the semigroup is uncountable then the rank is equal to the cardinality of
the semigroup. In this case, the notion of rank provides us with no
information. Here, we use a different rank property, namely the notation of
a relative rank. It was introduced by Ru\v skuc \cite{Ruskuc} in order to describe the generating sets of semigroups with infinite rank.
The \textit{relative rank} of a semigroup $S$ modulo a subset $A$ of $S$, denoted by $\rank (S : A)$,
is the minimal cardinality of a set $B\subseteq S$ such that $A \cup B$ generates $S$.
It follows immediately from the definition that $\rank (S : \emptyset) = \rank S$, $\rank ( S : S) = 0$,
$\rank (S : A ) = \rank (S : \langle A \rangle )$, and $\rank (S : A) = 0$ if and only if $A$ is a generating set for $S$.

Let $X$ be a nonempty infinite set and denote by $\T(X)$ the monoid of all full transformations of $X$ (under composition).
As a consequence of the main result in \cite{Sierp}, that every countable subset of $\T(X)$ is contained in a two-generated subsemigroup of $\T(X)$,
Banach obtained that the relative rank of $\T(X)$ modulo $A \subseteq \T(X)$ is either uncountable or at most two \cite{Ban}.
Relative ranks of subsemigroups of $\T(X)$ were considered in \cite{HHR, HRH}. Howie et al. showed that the relative rank of the full transformation semigroup $\T(X)$ modulo the set of all idempotent elements and modulo the symmetric group, respectively, is two. Moreover, they showed that the
relative rank of $\T(X)$ modulo the top $\mathcal{J}$-class of $\T(X)$ is zero.
This study was continued in \cite{HMR2} and was extended to subsemigroups of the monoid of all
binary relations and the symmetric inverse monoid (that is, all injective partial mappings) over infinite sets.
Similar properties were considered in the context of groups in \cite{DNT, Galv, McPhN}, where subgroups of symmetric groups are studied.
Recently, Tinpun and Koppitz considered generating sets of infinite full transformation semigroups with restricted range \cite{TK2}.

Now, let $X = (X,\leq)$ be a chain (totally ordered set). A function $f : A \rightarrow X$ from a subchain $A$ of $X$ into $X$ is said to be \textit{order-preserving} if $x \leq y$ implies $xf \leq yf$, for all $x, y \in A$. Notice that, given two subchains $A$ and $B$ of $X$ and an order-isomorphism (i.e. an order-preserving bijection) $f : A \rightarrow B$, then the inverse function $f^{-1} : B \rightarrow A$ is also an order-isomorphism. In this case, the subchains $A$ and $B$ are called \textit{order-isomorphic}. We denote by $\O(X)$ the submonoid of $\T(X)$ of all order-preserving transformations of $X$.
The relative rank of $\T(X)$ modulo $\O(X)$ was considered by Higgins, Mitchell, and Ru\v skuc \cite{HMR1}.
They showed that $\rank(\T(X) : \O(X)) = 1$, when $X$ is an arbitrary countable chain or an arbitrary well-ordered set,
while $\rank(\T(\R) : \O(\R))$ is uncountable, by considering the usual order of the real numbers $\R$. In \cite{HMMR}, the relative ranks $\rank(\T(X) : \O(X))$ were considered,
where $X$ is a countably infinite partially ordered set. In \cite{DFK}, Dimitrova, Fernandes, and Koppitz studied particular conditions for a countable infinite chain such that the monoid $\O(X)$ is generated by the set $J$ of its transformations with infinite range (image), i.e. $\rank (\O(X) : J) = 0$.

A generalization of the concept of an order-preserving transformation is the
concept of an orientation-preserving transformation, which was introduced in
1998 by McAlister \cite{McA} and, independently, one year later by Catarino and
Higgins \cite{CH}, but only for finite chains. In \cite{FJS}, Fernandes, Jesus, and Singha
introduced the concept of an orientation-preserving transformation on an
infinite chain. It generalizes the concept for a finite chain.
\begin{definition} \rm \cite{FJS}
Let $\alpha \in \T(X)$. We say that $\alpha$ is an \textit{orientation-preserving} transformation if there exists a non-empty subset $X_1$ of $X$ such that:\\
(1) $\alpha$ is order-preserving both on $X_1$ and on $X_2 = X\setminus X_1$;\\
(2) for all $a \in X_1$ and $b \in X_2$, we have $a < b$ and $a\alpha \geq b\alpha$.
\end{definition}
In the present paper, we will keep this meaning of $X_1$ and $X_2$ for an orientation-preserving transformation if it is clear from the
context. We call to such a subset $X_1$ an \textit{ideal} of $\alpha$. Notice that $X_1$ is an order ideal of $X$.
Recall that a subset $I$ of $X$ is called an order ideal of $X$ if $x \leq a$ implies $x \in I$, for all $x \in X$ and all $a \in I$.

Denote by $\OP(X)$ the subset of $\T(X)$ of all orientation-preserving transformations. In \cite{FJS}, Fernandes, Jesus, and Singha proved that $\OP(X)$ is a semigroup. Moreover, they proved that if $\alpha \in \OP(X)$ is a non-constant transformation then $\alpha$ admits a unique ideal.
Clearly, $\O(X) \subseteq \OP(X)$ and we have $\alpha \in \O(X)$ if and only if $\alpha \in \OP(X)$ and $\alpha$ admits $X$ as an ideal.

In contrast to the infinite case, the finite case is well studied. Several
properties of the monoid of orientation-preserving transformations on a
finite chain have been investigated in \cite{CH, McA}. A presentation
for this monoid, in terms of two (its rank) generators, was found by Arthur
and Ru\v{s}kuc \cite{AR}. The ranks of its ideals were determined by Zhao and Fernandes \cite{ZF}.
In \cite{TK1}, Tinpun and Koppitz considered the relative rank of the finite full transformation semigroup
with restricted range. In \cite{DKT}, Dimitrova, Koppitz, and Tinpun studied rank properties of the semigroup of orientation-preserving transformations with
restricted range on a finite chain. In \cite{DK}, Dimitrova and Koppitz determined the relative rank of the finite full transformation semigroup with restricted range $\T(X,Y)$ modulo its subsemigroup $\OP(X,Y)$ of all orientation-preserving transformation with restricted range. Moreover, they considered the relative rank of the semigroup $\OP(X,Y)$ modulo the monoid $\O(X,Y)$ of all order-preserving transformations in $\OP(X,Y)$.
The number of papers on full transformation semigroups and their subsemigroups for an infinite set $X$ is many times
smaller than in the finite case. In this paper, we determine the relative rank of the semigroup
$\OP(X)$ of all orientation-preserving transformations on infinite chains modulo the subsemigroup $\O(X)$ of all order-preserving transformations.
This corresponds to the second part of Problem 5.1 given by Fernandes, Jesus, and Singha in \cite{FJS}. We will give a complete answer for a certain class of
infinite chains, which contains classical infinite chains.

We begin by recalling some notations and definitions that will be used in the paper.
For every transformation $\alpha \in \T(X)$, we denote by $\dom \alpha$ and $\im \alpha$ the domain and the image (range) of $\alpha$, respectively.
The inverse of $\alpha$ is denoted by $\alpha^{-1}$. For a subset $A \subseteq \T(X)$, we denote by $\langle A \rangle$ the subsemigroup of $\T(X)$ generated by $A$.
For a subset $C \subseteq X$, we denote by $\alpha|_C$ the restriction of $\alpha$ to $C$ and by $\id_C$ the identity mapping on $C$.
A subset $C$ of $X$ is called a \textit{convex} subset of $X$ if $z \in X$ and $x < z < y$ imply $z \in C$, for all $x,y \in C$.
Let $C, D$ be convex subsets of $X$. We will write $C < D$ (respectively, $C \leq D$) if $c < d$ (respectively, $c \leq d$) for all $c \in C$ and all $d \in D$. If $C = \{c\}$ or $D = \{d\}$, we write $c < D$, $C < d$ (respectively, $c \leq D$, $C \leq d$) instead of $\{c\} < D$ or $C < \{d\}$ (respectively, $\{c\} \leq D$ or $C \leq \{d\}$).

For $A < B \subseteq X$ and $a, b \in X$ with $a < b$, $a < B$, and $A < b$, we put
$$(A,B) = \{x \in X : A < x < B\},$$
for a convex subset of $X$ which has no minimum and no maximum;
$$[a,B) = \{x \in X : a \leq x < B\},$$
for a convex subset of $X$ which has a minimum $a$ but no maximum;
$$(A,b] = \{x \in X : A < x \leq b\},$$
for a convex subset of $X$ which has no minimum but a maximum $b$;
$$[a,b] = \{x \in X : a \leq x \leq b\},$$
for a convex subset of $X$ which has a minimum $a$ and a maximum $b$.

Notice that if $A = \emptyset$ or $B = \emptyset$ then we have
$$(A,\emptyset) = \{x \in X : A < x\}, ~~~ (\emptyset,B) = \{x \in X : x < B\},$$
$$[a,\emptyset) = \{x \in X : a \leq x \}, ~\mbox{ and }~ (\emptyset,b] = \{x \in X : x \leq b\}.$$

A total order $<$ on a set $X$ is said to be dense if, for all $x, y \in X$ with $x < y$, there is a $z \in X$ such that $x < z < y$.
For the remaining part of the paper, $X$ will be an infinite densely totally ordered set and any two convex subsets of $X$ of the type $(A,B)$ (i.e. has no minimum and no maximum) are order-isomorphic. We will consider the following cases:
\begin{enumerate}
  \item $X$ has no minimum and no maximum.
  \item $X$ has a minimum and a maximum.
  \item $X$ has a minimum but no maximum.
  \item $X$ has no minimum but a maximum.
\end{enumerate}

The first case, when $X$ has no minimum and no maximum, was considered by Kittisak Tinpun in his thesis \cite{Tinp}. He proved that $\rank(\OP(X) : \O(X)) = 2$.
We will consider the other cases.\\
If $X$ has a minimum and a maximum, we put $a =\min X$ and $b =\max X$. Then we have $X = [a,b]$.\\
If $X$ has a minimum but no maximum, we put $a =\min X$. Then we have $X = [a,\emptyset)$.\\
If $X$ has no minimum but a maximum, we put $b =\max X$. Then we have $X = (\emptyset,b]$.

\begin{lemma}\label{le1} \rm
Let $X$ has a minimum $a$ and let $\alpha \in \OP(X)\setminus\O(X)$ with ideal $X_1$. Then $a\alpha$ is the minimum of $X_1\alpha$.
\end{lemma}
\begin{proof}
Since $a$ is the minimum of $X$, we have $a \leq c$ for all $c \in X = X_1 \cup X_2$. Moreover, $\alpha \in \OP(X)$ implies $\alpha |_{X_1}$ is order-preserving and thus, $a\alpha \leq c\alpha$, for all $c \in X_1$. Hence, $a\alpha$ is the minimum of $X_1\alpha$.
\end{proof}
~~\\

Dually, one can verify:
\begin{lemma}\label{le4} \rm
Let $X$ has a maximum $b$ and let $\alpha \in \OP(X)\setminus\O(X)$ with ideal $X_1$. Then $b\alpha$ is the maximum of $X_2\alpha$.
\end{lemma}

\begin{lemma}\label{le2} \rm
Let $\alpha \in \OP(X)\setminus \O(X)$ with ideal $X_1$ and let $X_1$ has a maximum or $X_2$ has a minimum. If $c \in X$ is the maximum of $X_1$ (respectively, the minimum of $X_2$) then $c\alpha$ is the maximum of $\im\alpha$ (respectively, $c\alpha$ is the minimum of $\im\alpha$).
\end{lemma}
\begin{proof}
Let $c \in X$ be the maximum of $X_1$ and let $y \in \im\alpha$. Then there is $x \in X$ such that $x\alpha = y$. If $x \in X_1$ then $x\alpha \leq c\alpha$, since $c$ is the maximum of $X_1$ and $\alpha |_{X_1}$ is order-preserving. If $x \in X_2$ then $x\alpha \leq z\alpha$, for all $z \in X_1$, since $X_2\alpha \leq X_1\alpha$ ($\alpha$ is orientation-preserving). Therefore, we obtain $y\leq c\alpha$, i.e. $c\alpha$ is the maximum of $\im\alpha$.

Now, let $c \in X$ be the minimum of $X_2$. Dually, we obtain that $c\alpha$ is the minimum of $\im\alpha$.
\end{proof}

\begin{proposition}\label{pr2} \rm \cite{FJS}
Let $\alpha \in \OP(X)$ with ideal $X_1$. If $X_1\alpha \cap X_2\alpha \neq \emptyset$ then $X_1\alpha \cap X_2\alpha = \{c\}$, for some $c \in X$. Moreover, in this case, $X_1\alpha$ has a minimum, $X_2\alpha$ has a maximum, and both of these elements coincide with $c$.
\end{proposition}

From Proposition \ref{pr2} and Lemma \ref{le1}, we obtain:
\begin{corollary}\label{co1} \rm
Let $X = [a,\emptyset)$ and let $\alpha \in \OP(X)\setminus\O(X)$ with ideal $X_1$. If $X_1\alpha \cap X_2\alpha \neq \emptyset$ then $X_1\alpha \cap X_2\alpha = \{a\alpha\}$.
\end{corollary}

From Proposition \ref{pr2} and Lemma \ref{le4}, we obtain:
\begin{corollary}\label{co4} \rm
Let $X = (\emptyset,b]$ and let $\alpha \in \OP(X)\setminus\O(X)$ with ideal $X_1$. If $X_1\alpha \cap X_2\alpha \neq \emptyset$ then $X_1\alpha \cap X_2\alpha = \{b\alpha\}$.
\end{corollary}

From Proposition \ref{pr2}, Lemma \ref{le1} and \ref{le4}, we obtain:
\begin{corollary}\label{co3} \rm
Let $X = [a,b]$ and let $\alpha \in \OP(X)\setminus\O(X)$ with ideal $X_1$. If $X_1\alpha \cap X_2\alpha \neq \emptyset$ then $a\alpha = b\alpha$ and $X_1\alpha \cap X_2\alpha = \{a\alpha\}$.
\end{corollary}

\section{$X$ Has a Minimum and a Maximum}

In this section, we consider the relative rank of the semigroup $\OP(X)$ modulo the semigroup $\O(X)$, whence the set $X$ has a minimum and a maximum. Recall that $a = \min X$, $b = \max X$, and $X = [a,b]$.

Let $c \in (a,b)$ and $d \in (c,b)$, i.e. $a < c < d < b$. Since any two convex subsets of $X$ of the type $(A,B)$ are order-isomorphic, we have that there are order-isomorphisms $\mu_1 : (a,c) \rightarrow (d,b)$ and $\mu_2 : (c,b) \rightarrow (a,c)$.

We define the transformation $\gamma : X \rightarrow X$ by
$$x\gamma = \left\{
             \begin{array}{ll}
               d, & x = a \\
               x\mu_1, & x \in (a,c) \\
               a, & x = c \\
               x\mu_2, & x \in (c,b) \\
               c, & x = b. \\
             \end{array}
           \right.$$
Since $\mu_1$ is an order-isomorphism and $a \leq y < c$ for all $y \in [a,c)$ as well as $d \leq y < b$ for all $y \in [d,b)$, we have that $\gamma|_{[a,c)}$ is order-preserving.
Since $\mu_2$ is an order-isomorphism and $c \leq y \leq b$ for all $y \in [c,b]$ as well as $a \leq y \leq c$ for all $y \in [a,c]$, we have that $\gamma|_{[c,b]}$ is order-preserving.
Moreover, for all $x_1 \in [a,c)$ and all $x_2 \in  [c,b]$, we have $x_1 < x_2$ and $x_2\gamma \leq c < d \leq x_1\gamma$ (since $x_2\gamma \in [a,c]$ and $x_1\gamma \in [d,b)$). Therefore, $\gamma \in \OP(X)$ with the ideal $X_1^* = [a,c)$. Clearly, $\gamma \notin \O(X)$ since $X_2^* = X \setminus X_1^* = [c,b] \neq \emptyset$.

\begin{theorem}\label{th1} \rm
$\rank(\OP(X) : \O(X)) = 1$.
\end{theorem}
\begin{proof}
Let $\alpha \in \OP(X)\setminus\O(X)$ with ideal $X_1$.
Moreover, from $X = [a,b]$, Lemma \ref{le1} and \ref{le4}, it follows $a\alpha = \min X_1\alpha$ and $b\alpha = \max X_2\alpha$.

We will consider two cases:
\begin{enumerate}
  \item[(a)] For any decomposition $X = \tilde{X}_1 \cup \tilde{X}_2$ with $\tilde{X}_1 < \tilde{X}_2$ holds $\tilde{X}_1$ has a maximum or $\tilde{X}_2$ has a minimum.
  \item[(b)] There is a decomposition $X = \tilde{X}_1 \cup \tilde{X}_2$ with $\tilde{X}_1 < \tilde{X}_2$ such that $\tilde{X}_1$ has no maximum and $\tilde{X}_2$ has no minimum.
\end{enumerate}
Notice that the case $\tilde{X}_1$ has a maximum and $\tilde{X}_2$ has a minimum is not possible for a decomposition $\tilde{X}_1 < \tilde{X}_2$, since $X$ is densely ordered set.

\textbf{Case (a).} There is $m \in X$ such that $X_1 = [a,m]$, $X_2 = (X_1,b] = (m,b]$ or $X_1 = [a,X_2) = [a,m)$, $X_2 = [m,b]$.\\

First, we will consider the case $X_1 = [a,m)$ and $X_2 = [m,b]$, i.e.
$$\alpha = \left(
            \begin{array}{ccc}
              X_1 & < & X_2 \\
              X_1\alpha & \geq & X_2\alpha
            \end{array}
          \right) = \left(
            \begin{array}{ccc}
              [a,m) & < & [m,b] \\
              ~[a,m)\alpha & \geq & ~[m,b]\alpha
            \end{array}
          \right).$$

Notice that if $m=b$ then $X_2 = [b,b] = \{b\}$.

Since any two convex subsets of $X$ of the type $(A,B)$ are order-isomorphic, we have that there are order-isomorphisms $\lambda_1 : (a,m) \rightarrow (a,c)$ and $\lambda_2 : (m,b) \rightarrow (c,b)$.

We define the transformation $\beta : X \rightarrow X$ by
$$x\beta = \left\{
             \begin{array}{ll}
               a, & x = a \neq m \\
               x\lambda_1, & x \in (a,m) \\
               c, & x = m \\
               x\lambda_2, & x \in (m,b) \\
               b, & x = b \neq m. \\
             \end{array}
           \right.$$
Clearly, $\beta$ is an order-preserving bijection on $X$, i.e. $\beta \in \O(X)$.

Further, we define the transformation $\delta : X \rightarrow X$ by
$$x\delta = \left\{
             \begin{array}{ll}
               m\alpha, & x = a \\
               x\mu_2^{-1}\lambda_2^{-1}\alpha, & x \in (a,c) \\
               b\alpha, & x = c \\
               a\alpha, & x \in (c,d] \\
               x\mu_1^{-1}\lambda_1^{-1}\alpha, & x \in (d,b) \\
               b, & x = b. \\
             \end{array}
           \right.$$
Note that $\alpha \in \OP(X)\setminus\O(X)$ implies $X_2\alpha = [m,b]\alpha \leq X_1\alpha = [a,m)\alpha$. From Lemma \ref{le2}, we have $m\alpha = \min \im\alpha$. Thus, $m\alpha = \min X_2\alpha \leq (m,b)\alpha \leq b\alpha = \max X_2\alpha \leq a\alpha = \min X_1\alpha \leq (a,m)\alpha \leq b = \max X$. Moreover, $\lambda_i, \mu_i$, $i=1,2$, are order-isomorphisms and $\alpha |_{(m,b)}$, $\alpha |_{(a,m)}$ are order-preserving imply $\delta|_{(a,c)}$ and $\delta|_{(d,b)}$ are order-preserving.
Therefore, we obtain $\delta \in \O(X)$.

Now, we will show that $\alpha = \beta\gamma\delta$. Let $x \in X$.\\
If $x = a$ then $a\beta\gamma\delta = a\gamma\delta = d\delta = a\alpha$.\\
If $x \in (a,m)$ then $x\beta\gamma\delta = (x\lambda_1)\gamma\delta = (x\lambda_1\mu_1)\delta = x\lambda_1\mu_1\mu_1^{-1}\lambda_1^{-1}\alpha = x\alpha$, since $\mu_1\mu_1^{-1} = \id_{(a,c)}$, $\lambda_1\id_{(a,c)}\lambda_1^{-1} = \id_{(a,m)}$, and $x\id_{(a,m)}\alpha = x\alpha$.\\
If $x = m$ then $m\beta\gamma\delta = c\gamma\delta = a\delta = m\alpha$.\\
If $x \in (m,b)$ then $x\beta\gamma\delta = (x\lambda_2)\gamma\delta = (x\lambda_2\mu_2)\delta = x\lambda_2\mu_2\mu_2^{-1}\lambda_2^{-1}\alpha = x\alpha$, since $\mu_2\mu_2^{-1} = \id_{(c,b)}$, $\lambda_2\id_{(c,b)}\lambda_2^{-1} = \id_{(m,b)}$, and $x\id_{(m,b)}\alpha = x\alpha$.\\
If $x = b$ then $b\beta\gamma\delta = b\gamma\delta = c\delta = b\alpha$.

Therefore, we have $\alpha \in \langle \O(X),\gamma \rangle$.\\

Now, consider the case $X_1 = [a,m]$ and $X_2 = (m,b]$, i.e.
$$\alpha = \left(
            \begin{array}{ccc}
              X_1 & < & X_2 \\
              X_1\alpha & \geq & X_2\alpha
            \end{array}
          \right) = \left(
            \begin{array}{ccc}
              [a,m] & < & (m,b] \\
              ~[a,m]\alpha & \geq & ~(m,b]\alpha
            \end{array}
          \right).$$

Notice that if $a=m$ then $X_1 = [a,a] = \{a\}$.

Let $c' \in (a,c)$, i.e. $(a,c') \subseteq X_1^* = [a,c)$. Recall that $c < d < b$, i.e. $d \in X_2^*$ and $(c,d) \subseteq X_2^* = [c,b]$.

Since any two convex subsets of $X$ of the type $(A,B)$ are order-isomorphic, we have that there is an order-isomorphism $\tau_1 : (a,c) \rightarrow (a,c')$.

We define the transformation $\eta_1 : X \rightarrow X$ by
$$x\eta_1 = \left\{
             \begin{array}{ll}
               a, & x = a \\
               x\tau_1, & x \in (a,c) \\
               c', & x = c \\
               x, & x \in (c,b]. \\
             \end{array}
           \right.$$
Since $a\eta_1 = a = \min X < (a,c)\eta_1 = (a,c)\tau_1 = (a,c') < c' = c\eta_1 < (c,b) < b$, $\eta_1|_{(c,b]} = \id_{(c,b]}$, and $\tau_1$ is an order-isomorphism, we have that $\eta_1$ is an order-preserving transformation on $X$, i.e. $\eta_1 \in \O(X)$.

Further, since any two convex subsets of $X$ of the type $(A,B)$ are order-isomorphic, we have that there is an order-isomorphism $\tau_2 : (d,c'\mu_1) \rightarrow (d,b)$.

We define the transformation $\eta_2 : X \rightarrow X$ by
$$x\eta_2 = \left\{
             \begin{array}{ll}
               x, & x \in [a,c] \\
               d, & x \in (c,d] \\
               x\tau_2, & x \in (d,c'\mu_1) \\
               b, & x \in [c'\mu_1,b]. \\
             \end{array}
           \right.$$
Since $\eta_2|_{[a,c]} = \id_{[a,c]}$, $c < (c,d]\eta_2 = d < (d,c'\mu_1)\eta_2 = (d,c'\mu_1)\tau_2 = (d,b) < b = [c'\mu_1,b]\eta_2$, and $\tau_2$ is an order-isomorphism, we have that $\eta_2$ is an order-preserving transformation on $X$, i.e. $\eta_2 \in \O(X)$.

Since $\mu_1$ is an order-isomorphism, $x\gamma = x\mu_1$, for all $x \in (a,c)$, and $c' \in (a,c)$, we have $c'\gamma = c'\mu_1 \in (d,b)$ and thus $d < c'\gamma < b$. Therefore, we can write the transformation $\gamma$ in the following way:
$$\gamma =
\left(
  \begin{array}{ccccccccccccc}
    a & < & (a,c') & < & c' & < & (c',c) & < & c & < & (c,b) & < & b \\
    d & < & (a,c')\mu_1 & < & c'\mu_1 & < & (c',c)\mu_1 & > & a & < & (c,b)\mu_2 & < & c
  \end{array}
\right).$$

Further, we define the transformation $\delta_1 : X \rightarrow X$ by
$$x\delta_1 = \left\{
             \begin{array}{ll}
               a, & x = a \\
               x\mu_2^{-1}\lambda_2^{-1}\alpha, & x \in (a,c) \\
               b\alpha, & x = c \\
               a\alpha, & x \in (c,d] \\
               x\tau_2^{-1}\mu_1^{-1}\tau_1^{-1}\lambda_1^{-1}\alpha, & x \in (d,b) \\
               m\alpha, & x = b. \\
             \end{array}
           \right.$$
Since $\lambda_1, \lambda_2, \mu_1, \mu_2, \tau_1, \tau_2$ are order-isomorphisms, $a = \min X$, $\alpha |_{(m,b]}$ is order-preserving ($X_2=(m,b]$), $b\alpha = \max X_2\alpha$ (by Lemma \ref{le4}), $a\alpha = \min X_1\alpha$ (by Lemma \ref{le1}), $X_2\alpha \leq X_1\alpha$ ($\alpha \in \OP(X)$), $\alpha |_{[a,m]}$ is order-preserving ($X_1 = [a,m]$ is the ideal of $\alpha$), and $m\alpha = \max \im\alpha$ (by Lemma \ref{le2}), i.e. $\min X = a \leq X_2\alpha \leq b\alpha = \max X_2\alpha \leq a\alpha = \min X_1\alpha \leq X_1\alpha \leq m\alpha = \max \im\alpha$, it follows that $\delta_1 \in \O(X)$.

Finally, we will show that $\alpha = \beta\eta_1\gamma\eta_2\delta_1$. Let $x \in X$.\\
If $x = a \neq m$ then $a\beta\eta_1\gamma\eta_2\delta_1 = a\eta_1\gamma\eta_2\delta_1 = a\gamma\eta_2\delta_1 = d\eta_2\delta_1 = d\delta_1 = a\alpha$.\\
If $x \in (a,m)$ then $x\beta\eta_1\gamma\eta_2\delta_1 = (x\lambda_1)\eta_1\gamma\eta_2\delta_1 = (x\lambda_1\tau_1)\gamma\eta_2\delta_1 = (x\lambda_1\tau_1\mu_1)\eta_2\delta_1 = (x\lambda_1\tau_1\mu_1\tau_2)\delta_1 = x\lambda_1\tau_1\mu_1\tau_2\tau_2^{-1}\mu_1^{-1}\tau_1^{-1}\lambda_1^{-1}\alpha$.
Since $\tau_2\tau_2^{-1} = \id_{(d,c'\mu_1)}$,\\
$\mu_1|_{(a,c')}\id_{(d,c'\mu_1)}\mu_1^{-1} = \id_{(a,c')}$, $\tau_1\id_{(a,c')}\tau_1^{-1} = \id_{(a,c)}$, and $\lambda_1\id_{(a,c)}\lambda_1^{-1} = \id_{(a,m)}$,
we have $x\lambda_1\tau_1\mu_1\tau_2\tau_2^{-1}\mu_1^{-1}\tau_1^{-1}\lambda_1^{-1}\alpha = x\id_{(a,m)}\alpha = x\alpha$.\\
If $x = m$ then $m\beta\eta_1\gamma\eta_2\delta_1 = c\eta_1\gamma\eta_2\delta_1 = c'\gamma\eta_2\delta_1 = (c'\mu_1)\eta_2\delta_1 = b\delta_1 = m\alpha$.\\
If $x \in (m,b)$ then $x\beta\eta_1\gamma\eta_2\delta_1 = (x\lambda_2)\eta_1\gamma\eta_2\delta_1 = (x\lambda_2)\gamma\eta_2\delta_1 = (x\lambda_2\mu_2)\eta_2\delta_1 = (x\lambda_2\mu_2)\delta_1 = x\lambda_2\mu_2\mu_2^{-1}\lambda_2^{-1}\alpha = x\alpha$, since $\mu_2\mu_2^{-1} = \id_{(c,b)}$, $\lambda_2\id_{(c,b)}\lambda_2^{-1} = \id_{(m,b)}$, and $x\id_{(m,b)}\alpha = x\alpha$.\\
If $x = b$ then $b\beta\eta_1\gamma\eta_2\delta_1 = b\eta_1\gamma\eta_2\delta_1 = b\gamma\eta_2\delta_1 = c\eta_2\delta_1 = c\delta_1 = b\alpha$.

Therefore, we have $\alpha \in \langle \O(X),\gamma \rangle$.\\

\textbf{Case (b).} There is a decomposition $X = X_1 \cup X_2$ with $X_1 < X_2$ such that $X_1$ has no maximum and $X_2$ has no minimum.
Let $\alpha \in \OP(X)\setminus \O(X)$ with ideal $X_1$. Then $X_1 = [a,X_2)$ and $X_2 = (X_1,b]$, i.e.
$$\alpha = \left(
            \begin{array}{ccc}
              X_1 & < & X_2 \\
              X_1\alpha & \geq & X_2\alpha
            \end{array}
          \right) = \left(
            \begin{array}{ccc}
              [a,X_2) & < & (X_1,b] \\
              ~[a,X_2)\alpha & \geq & ~(X_1,b]\alpha
            \end{array}
          \right).$$

Since any two convex subsets of $X$ of the type $(A,B)$ are order-isomorphic, we have that there are order-isomorphisms $\lambda_3 : (a,X_2) \rightarrow (a,c)$ and $\lambda_4 : (X_1,b) \rightarrow (c,b)$.

We define the transformation $\theta_1 : X \rightarrow X$ by
$$x\theta_1 = \left\{
             \begin{array}{ll}
               a, & x = a \\
               x\lambda_3, & x \in (a,X_2) \\
               x\lambda_4, & x \in (X_1,b) \\
               b, & x = b. \\
             \end{array}
           \right.$$
Clearly, $\theta_1$ is an order-preserving transformation on $X$, i.e. $\theta_1 \in \O(X)$.

Recall that $c < d < b$, i.e. $d \in X_2^*$ and $(c,d) \subseteq X_2^* = [c,b]$.
We define the transformation $\theta_2 : X \rightarrow X$ by
$$x\theta_2 = \left\{
             \begin{array}{ll}
               a, & x = a \\
               x\mu_2^{-1}\lambda_4^{-1}\alpha, & x \in (a,c) \\
               b\alpha, & x = c \\
               a\alpha, & x \in (c,d] \\
               x\mu_1^{-1}\lambda_3^{-1}\alpha, & x \in (d,b) \\
               b, & x = b. \\
             \end{array}
           \right.$$
Since $\mu_1$, $\mu_2$, $\lambda_3$, $\lambda_4$ are order-isomorphisms, $a = \min X \leq (X_1,b)\alpha \leq b\alpha \leq a\alpha \leq (a,X_2)\alpha \leq b$, $\alpha |_{(X_1,b]}$ is order-preserving ($X_2=(X_1,b]$), and $\alpha |_{[a,X_2)}$ is order-preserving ($X_1 = [a,X_2)$ is the ideal of $\alpha$), it follows that $\theta_2 \in \O(X)$.

Now, we will show that $\alpha = \theta_1\gamma\theta_2$. Let $x \in X$.\\
If $x = a$ then $a\theta_1\gamma\theta_2 = a\gamma\theta_2 = d\theta_2 = a\alpha$.\\
If $x \in (a,X_2)$ then $x\theta_1\gamma\theta_2 = (x\lambda_3)\gamma\theta_2 = (x\lambda_3\mu_1)\theta_2 = x\lambda_3\mu_1\mu_1^{-1}\lambda_3^{-1}\alpha = x\alpha$, since $\mu_1\mu_1^{-1} = \id_{(a,c)}$, $\lambda_3\id_{(a,c)}\lambda_3^{-1} = \id_{(a,X_2)}$, and $x\id_{(a,X_2)}\alpha = x\alpha$.\\
If $x \in (X_1,b)$ then $x\theta_1\gamma\theta_2 = (x\lambda_4)\gamma\theta_2 = (x\lambda_4\mu_2)\theta_2 = x\lambda_4\mu_2\mu_2^{-1}\lambda_4^{-1}\alpha = x\alpha$, since $\mu_2\mu_2^{-1} = \id_{(c,b)}$, $\lambda_4\id_{(c,b)}\lambda_4^{-1} = \id_{(X_1,b)}$, and $x\id_{(X_1,b)}\alpha = x\alpha$.\\
If $x = b$ then $b\theta_1\gamma\theta_2 = b\gamma\theta_2 = c\theta_2 = b\alpha$.\\
Therefore, we have $\alpha \in \langle \O(X),\gamma \rangle$.\\

Altogether, we obtain $\OP(X) = \langle \O(X),\gamma \rangle$ and since $\O(X)$ is a proper submonoid of $\OP(X)$ we have $\rank(\OP(X) : \O(X)) = 1$.
\end{proof}

\section{$X$ Has a Minimum but no Maximum}

In this section, we consider the relative rank of the semigroup $\OP(X)$ modulo the semigroup $\O(X)$, whence the set $X$ has a minimum $a$ but no maximum, i.e. $X = [a,\emptyset)$.

Let $b \in X$, with $a < b$ and let $Y = [a,b] \subset X$. Then $\OP(Y)$ is the semigroup of all orientation-preserving transformations of $Y$ and by Theorem \ref{th1}, we have $\OP(Y) = \langle \O(Y), \gamma \rangle$.
In fact, we will assume that $\gamma$ is defined as in Section 2, where the set $X$ is replaced
by the set $Y$ and $d \in Y$ is given. Recall also that
$\mu_1 : (a,c) \rightarrow (d,b)$ and $\mu_2 : (c,b) \rightarrow (a,c)$ are order-isomorphisms.

Let us put
$$\OP^*(X) = \{\beta \in \OP(X) : \im \beta \subseteq [a,b] \mbox{ and } |[b,\emptyset)\beta| = 1\}.$$

\begin{lemma}\label{le6} \rm
$\OP^*(X)$ is isomorphic to $\OP(Y)$.
\end{lemma}
\begin{proof}
We consider the mapping $h : \OP(Y) \rightarrow \OP^*(X)$, defined by
$$xh(\alpha) = \left\{
             \begin{array}{ll}
               x\alpha, & x \in [a,b) \\
               b\alpha, & x \in [b,\emptyset), \\
             \end{array}
           \right.$$
for $\alpha \in \OP(Y)$.
Since $\beta \in \OP^*(X)$ is uniquely determined by the image of $Y$ and
$xh(\alpha) = x\alpha$ for all $x \in Y = [a,b]$, we can conclude that $h$ is a bijection.

Now, let $\alpha, \beta \in \OP(Y)$. We will show that $h(\alpha\beta) = h(\alpha)h(\beta)$. Since $\alpha \in \OP(Y)$, we have $x\alpha \in Y = [a,b]$ for all $x \in Y$.
Let $x \in [a,b)$.
Then $xh(\alpha\beta) = x(\alpha\beta) = (x\alpha)\beta = (x\alpha)h(\beta) = xh(\alpha)h(\beta)$.
Let $x \in [b,\emptyset)$.
Then $xh(\alpha\beta) = b(\alpha\beta) = (b\alpha)\beta = (b\alpha)h(\beta) = xh(\alpha)h(\beta)$.

Therefore, $h$ is an isomorphism and thus, $\OP^*(X)$ is isomorphic to $\OP(Y)$.
\end{proof}

Let $\O^*(X) = \OP^*(X) \cap \O(X)$. Then from Theorem \ref{th1} and Lemma \ref{le6}, we have:
\begin{corollary}\label{co6} \rm
  $\OP^*(X) = \langle \O^*(X), h(\gamma) \rangle$.
\end{corollary}

Since any two convex subsets of $X$ of the type $(A,B)$ are order-isomorphic, we have that there is an order-isomorphism $\nu : (a,c) \rightarrow (c,\emptyset)$.

We define the transformation $\gamma^* : X \rightarrow X$ by
$$x\gamma^* = \left\{
             \begin{array}{ll}
               c, & x = a \\
               x\nu, & x \in (a,c) \\
               a, & x = c \\
               x\nu^{-1}, & x \in (c,\emptyset). \\
             \end{array}
           \right.$$
Since $\nu$ is an order-isomorphism, $c < x\nu \in (c,\emptyset)$, and $a < x\nu^{-1} \in (a,c)$, we have that $\gamma^*|_{[a,c)}$ and
$\gamma^*|_{[c,\emptyset)}$ are order-preserving.
Moreover, for all $x_1 \in [a,c)$ and $x_2 \in  [c,\emptyset)$, we have $x_1 < x_2$ and $x_2\gamma^* < c \leq x_1\gamma^*$ (since $x_2\gamma^* \in [a,c)$ and $x_1\gamma^* \in [c,\emptyset)$). Therefore, $\gamma^* \in \OP(X)\setminus\O(X)$ with the same ideal $[a,c)$ as the transformation $\gamma \in \OP(Y)\setminus\O(Y)$. Moreover, $\gamma^*$ is an injective orientation-preserving transformation on $X$ and $\im \gamma^* = [a,c) \cup [c,\emptyset) = X$.

\begin{proposition}\label{pr3} \rm
$\OP^*(X) \subseteq \langle \O(X), \gamma^* \rangle$.
\end{proposition}

\begin{proof}
Since $\OP^*(X) = \langle \O^*(X), h(\gamma) \rangle$ (by Corollary \ref{co6}) and $\O^*(X) \subseteq \O(X)$, it remains to show that $h(\gamma) \in \langle \O(X), \gamma^* \rangle$.

We define the transformation $\delta: X \rightarrow X$ by
$$x\delta = \left\{
             \begin{array}{ll}
               a, & x = a \\
               x\nu\mu_2, & x \in (a,c) \mbox{ and } x\nu \in (c,b)\\
               c, & x \in (a,c) \mbox{ and } x\nu \in [b,\emptyset)\\
               d, & x = c \\
               x\nu^{-1}\mu_1, & x \in (c,\emptyset). \\
             \end{array}
           \right.$$
Since $\nu, \mu_1, \mu_2$ are order-isomorphisms, $x\nu\mu_2 \in (a,c)$, for $x \in (a,c)$, $x\nu \in (c,b)$, and $x\nu^{-1}\mu_1 \in (d,b)$, for $x \in (c,\emptyset)$, i.e. $$a <  x\nu\mu_2 < c < d < x\nu^{-1}\mu_1,$$ it follows that $\delta \in \O(X)$.

Now, we will show that $h(\gamma) = \gamma^*\delta$. Let $x \in X$.\\
If $x = a$ then $a\gamma^*\delta = c\delta = d = a\gamma = ah(\gamma)$.\\
If $x \in (a,c)$ then $x\gamma^*\delta = x\nu\delta = x\nu\nu^{-1}\mu_1 = x\mu_1 = x\gamma = xh(\gamma)$, since $\nu\nu^{-1} = \id_{(a,c)}$ and $x\id_{(a,c)}\mu_1 = x\mu_1$.\\
If $x = c$ then $c\gamma^*\delta = a\delta = a = c\gamma = ch(\gamma)$.\\
If $x \in (c,b)$ then $x\gamma^*\delta = x\nu^{-1}\delta = x\nu^{-1}\nu\mu_2 = x\mu_2 = x\gamma = xh(\gamma)$, since $(\nu^{-1}\nu)|_{(c,b)} = \id_{(c,b)}$ and $x\id_{(c,b)}\mu_2 = x\mu_2$.\\
If $x \in [b,\emptyset)$ then $x\gamma^*\delta = x\nu^{-1}\delta = c = b\gamma = xh(\gamma)$.

Therefore, we obtain $h(\gamma) = \gamma^*\delta \in \langle \O(X),\gamma^* \rangle$ and thus, $\OP^*(X) = \langle \O^*(X), h(\gamma) \rangle \subseteq \langle \O(X), \gamma^* \rangle$.
\end{proof}

\begin{theorem}\label{th2} \rm
$\rank(\OP(X) : \O(X)) = 1$.
\end{theorem}
\begin{proof}
Let $\beta \in \OP(X)\setminus\O(X)$.

Since any two convex subsets of $X$ of the type $(A,B)$ are order-isomorphic, we have that there is an order-isomorphism $\tau_1 : (a,b) \rightarrow (a,\emptyset)$.

We define the transformation $\eta_1 : X \rightarrow X$ by
$$x\eta_1 = \left\{
             \begin{array}{ll}
               a, & x = a \\
               x\tau_1, & x \in (a,b) \\
               a, & x \in [b,\emptyset). \\
             \end{array}
           \right.$$
Clearly, $\eta_1$ is an orientation-preserving transformation of $X$ with ideal $[a,b)$, i.e. $\eta_1 \in \OP(X)$.

Further, let $I \subseteq X$ be the smallest convex subset of $X$ containing $\im \beta$, i.e. $\im \beta \subseteq I = \bigcap\{H : H$ is a convex subset of $X$ with $\im\beta \subseteq H\}$. Let $Y' \subseteq Y = [a,b]$ be such that
\begin{description}
  \item[-] if $I$ has a minimum and a maximum then $Y' = [a,b]$;
  \item[-] if $I$ has a minimum but no maximum then $Y' = [a,b)$;
  \item[-] if $I$ has no minimum but a maximum then $Y' = (a,b]$;
  \item[-] if $I$ has no minimum and no maximum then $Y' = (a,b)$.
\end{description}

Then there is an order-isomorphism $\tau_2 : I \rightarrow Y'$ and we define the transformation $\eta_2 : X \rightarrow X$ by
$$x\eta_2 = \left\{
             \begin{array}{ll}
               a, & x \in [a,I) \\
               x\tau_2, & x \in I \\
               b, & x \in (I,\emptyset). \\
             \end{array}
           \right.$$
Since $\tau_2$ is an order-isomorphism and $a \leq x\tau_2 \in Y' \subseteq [a,b] \leq b$, we have that
$\eta_2$ is an order-preserving transformation of $X$, i.e. $\eta_2 \in \O(X)$.
Note that if $\im \beta = X$ then $I = X$, $Y' = [a,b)$, and $a\eta_2 = a\tau_2 = a$, $(a,\emptyset)\eta_2 = (a,\emptyset)\tau_2 = (a,b) = (a,\emptyset)\tau_1^{-1}$. Moreover, in the case $I = X$ we have $\tau_2 \in \O(X)$.

Now, we consider the transformation $\beta^* = \eta_1\beta\eta_2$. Since $\beta, \eta_1, \eta_2 \in \OP(X)$, it follows that $\beta^* \in \OP(X)$.
We will show that $\beta^* \in \OP^*(X)$. Let $x \in X$. Then $(x\eta_1)\beta \in \im \beta \subseteq I$ and thus, $(x\eta_1\beta)\eta_2 = (x\eta_1\beta)\tau_2 \in Y' \subseteq [a,b]$. Thus, $x\beta^* = x\eta_1\beta\eta_2 \in [a,b]$.
Hence, $\im \beta^* = \im (\eta_1\beta\eta_2) \subseteq [a,b]$. Moreover, from $[b,\emptyset)\beta^* = \{a\beta\tau_2\}$, it follows $|[b,\emptyset)\beta^*| = 1$.
Therefore, $\beta^* \in \OP^*(X)$. From Proposition \ref{pr3}, it follows that $\beta^* \in \langle \O(X), \gamma^* \rangle$.

Further, we define the transformation $\theta_1 : X \rightarrow X$ by
$$x\theta_1 = \left\{
             \begin{array}{ll}
               a, & x = a \\
               x\tau_1^{-1}, & x \in (a,\emptyset). \\
             \end{array}
           \right.$$
Since $\tau_1$ is an order-isomorphism and $a < x\tau_1^{-1} \in (a,b)$, we have that
$\theta_1$ is an order-preserving transformation of $X$, i.e. $\theta_1 \in \O(X)$.

Now, we consider two cases for the set $I$.

1) If $(I,\emptyset) \neq \emptyset$ then we define the transformation  $\theta_2 : X \rightarrow X$ by
$$x\theta_2 = \left\{
             \begin{array}{ll}
               a, & x = a \mbox{ and } a \notin Y' \\
               x\tau_2^{-1}, & x \in Y' \\
               p, & x \in (Y',\emptyset), \\
             \end{array}
           \right.$$
where $p \in (I,\emptyset)$. Since $\tau_2$ is an order-isomorphism and $a \leq Y'\tau_2^{-1} = I < p \in (I,\emptyset)$, we obtain $\theta_2 \in \O(X)$.

2) If $(I,\emptyset)=\emptyset$ then there is an order-isomorphism $\eta_3$
from $Y'$ into a subset of $[a,c)$ with $(Y'\eta_3,c)=\emptyset$. We define a transformation $\theta_{2,1} : X \rightarrow X$ by
$$x\theta_{2,1} = \left\{
\begin{array}{ll}
a, &  x < Y' \\
x\eta_3, & x \in Y' \\
c, &  x > Y'.
\end{array}
\right.$$
Since $\eta_3$ is an order-isomorphism and $a \leq Y'\eta_3 < c$, we have $\theta_{2,1}\in \O(X)$.

Further, we define $\theta_{2,2} : X \rightarrow X$ by
$$x\theta_{2,2} = \left\{
\begin{array}{ll}
a, & x < Y'\eta_3\gamma^* \\
x(\eta_3\gamma^*)^{-1}\tau_{2}^{-1}, & x \in Y'\eta_3\gamma^*.
\end{array}
\right.$$
Since $(Y'\eta_3,c)=\emptyset$ and $\gamma^*|_{[a,c)}$ is an order-isomorphism, we have ($Y'\eta_3\gamma^*,\emptyset)=\emptyset$.
Hence, $\theta_{2,2}$ is well defined. Since $\eta_3$ as well as $\gamma^*|_{[a,c)}$ are order-isomorphisms
and $a \leq I = Y'\tau_{2}^{-1} = Y'(\eta_3\gamma^*)(\eta_{3}\gamma^*)^{-1}\tau_{2}^{-1}$, we have that $\theta_{2,2} \in \O(X)$.

We put $\theta_2 = \theta_{2,1}\gamma^*\theta_{2,2} \in \left\langle \O(X),\gamma^* \right\rangle$.
Note that $\theta_2|_{Y'} = (\theta_{2,1}\gamma^*\theta_{2,2})|_{Y'} = \eta_3\gamma^*(\eta_3\gamma^*)^{-1}\tau_2^{-1} = \tau_2^{-1}$.\\

Finally, we will show that $\beta = \theta_1\beta^*\theta_2$. Let $x \in X$.\\
If $x = a$ then $a\theta_1\beta^*\theta_2 = a\theta_1\eta_1\beta\eta_2\theta_2 = a\eta_1\beta\eta_2\theta_2 = a\beta\eta_2\theta_2 = a\beta\tau_2\theta_2 =  a\beta\tau_2\tau_2^{-1} = a\beta$, since $\tau_2\tau_2^{-1} = \id_I$, $\im \beta \subseteq I$. \\
If $x \in (a,\emptyset)$ then $x\theta_1\beta^*\theta_2 = x\tau_1^{-1}\beta^*\theta_2 = x\tau_{1}^{-1}\eta_{1}\beta\eta_{2}\theta_{2} = x\tau_{1}^{-1}\tau_{1}\beta\eta_{2}\theta_{2} = x\tau_{1}^{-1}\tau_{1}\beta\tau_{2}\theta_{2} = x\tau_{1}^{-1}\tau_{1}\beta\tau_{2}\tau_{2}^{-1}$, since $\im\beta \subseteq I$. Because $\tau_{1}^{-1}\tau_{1} = \id_{(a,\emptyset)}$, $\tau_{2}\tau_{2}^{-1} = \id_{I}$, and $x\id_{(a,\emptyset)}\beta \id_{I} = x\beta$, we obtain
$x\tau_{1}^{-1}\tau_{1}\beta\tau_{2}\tau_{2}^{-1} = x\beta$, i.e. $x\theta_1\beta^*\theta_2 = x\beta$.\\
Therefore, we have $\beta = \theta_1\beta^*\theta_2 \in \langle \O(X), \gamma^* \rangle$.\\

Altogether, we obtain $\OP(X) = \langle \O(X),\gamma^* \rangle$ and since $\O(X)$ is a proper submonoid of $\OP(X)$, we have $\rank(\OP(X) : \O(X)) = 1$.
\end{proof}
~~\\

One can dually consider the case $X$ has no minimum but a maximum. In this case, one obtains also $\rank(\OP(X) : \O(X)) = 1$.

\end{document}